\documentclass[12pt]{amsart}

\usepackage{enumerate}
\usepackage{latexsym}
\usepackage{amssymb, amsmath, amsthm}
\usepackage{subfigure}
\usepackage{enumitem}
\usepackage{hyperref}
\usepackage{xcolor}
\usepackage{comment}

\makeatletter
\@namedef{subjclassname@2010}{%
  \textup{2010} Mathematics Subject Classification}
\makeatother

\newtheorem{theorem}{Theorem}[section]
\newtheorem{corollary}[theorem]{Corollary}
\newtheorem{lemma}[theorem]{Lemma}
\newtheorem{proposition}[theorem]{Proposition}
\newtheorem{claim}[theorem]{Claim}

\theoremstyle{definition}

\newtheorem{remark}[theorem]{Remark}

\numberwithin{equation}{section}

\newcommand{\f}{\mathcal{F}}
\newcommand{\g}{\mathcal{G}}

\newcommand{\K}{\mathcal{K}}

\newcommand{\aut}{{\rm Aut}{(\lel)}}
\newcommand{\lel}{\mathbb{L}}
\newcommand{\ian}{\mathbb{I}}
\newcommand{\fan}{\mathbb{F}}

\newcommand{\x}{\mathbb{X}}

\begin{document}
\baselineskip=17pt

\title{Lelek fan from a projective Fra\"{i}ss\'{e} limit}

\author[D. Barto\v{s}ov\'{a}]{Dana Barto\v{s}ov\'{a}}
\address{Institute de Matematica e Estat\'istica, Universidade de S\~ao Paulo, Brazil}
\email{dana@ime.usp.br}

\author[A. Kwiatkowska]{Aleksandra Kwiatkowska}
\address{Department of Mathematics, University of California, Los Angeles,~USA}
\email{akwiatk2@math.ucla.edu}

\date{}

\begin{abstract}
We show that a natural quotient of the projective Fra\"{i}ss\'{e} limit of a family  that consists of finite rooted trees is the Lelek fan.
Using this construction, we study properties of the Lelek fan and of its homeomorphism group. We show that the Lelek fan is projectively universal and projectively ultrahomogeneous in the class of smooth fans.
We further show that the homeomorphism group of the Lelek fan is totally disconnected, generated by every neighbourhood of the identity, has a dense conjugacy class, and is simple.
\end{abstract}

\subjclass[2010]{03E15, 37B05, 54F15, 03C98}

\keywords{Lelek fan, Fra\"{i}ss\'{e} limits, homeomorphism groups}

\maketitle

\section{Introduction}

\subsection{Lelek fan}\label{notation}
 A {\em continuum} is a compact and connected metric space.
Let $C$ denote the Cantor set. The {\em Cantor fan} $F$ is the cone over the Cantor set, that is
$C\times [0,1]/\!\!\sim$, where $(a,b)\sim (c,d)$ if and only if $a=c$ and $b=d$ or $b=d=0$. Recall that an {\em arc} is a homeomorphic image of the closed unit interval $[0,1].$ If $X$ is a space and $h:[0,1]\to X$ is a homeomorphism onto its image, we call $h(0)=a$ and $h(1)=b$ the endpoints of the arc given by $h$ and denote this arc as $ab$.
An {\em endpoint} of a continuum $X$ is a point  $e$ such that for every arc $ab$ in $X$, if $e\in ab$, then $e=a$ or $e=b$.
Finally, a {\em  Lelek fan} $L$ is a  non-degenerate subcontinuum of the Cantor fan with a dense set of endpoints.

In the literature,  a Lelek fan is often defined as a smooth fan with a dense set of endpoints. 
However, smooth fans  are exactly  fans that can be embedded into the Cantor fan
 (see \cite{CC}, Proposition 4). We give the definition of a smooth fan in Subsection \ref{pu}.
 

 A Lelek fan was constructed by Lelek  \cite{L}.
Several characterizations of a Lelek fan  were collected in \cite{CCM}, Theorem 12.14.
A remarkable property of a Lelek fan is its uniqueness, which was proved independently by Bula-Oversteegen \cite{BO} and 
by W. Charatonik \cite{C}: any two non-degenerate subcontinua of the Cantor fan with a dense set of endpoints are homeomorphic.
We can therefore speak about ``the'' Lelek fan.

A very interesting and well-studied by many people is the space $E$ of endpoints of the Lelek fan $L$. 
This space is a dense $G_\delta$ set in $L$, therefore it
 is separable and completely metrizable. 
It is homeomorphic to the complete Erd\H{o}s space, to
the set of endpoints of the Julia set of the exponential map, to the set of endpoints of the separable universal $\mathbb{R}$-tree;
see Kawamura-Oversteegen-Tymchatyn \cite{KOT} for more details. Since the complete Erd\H{o}s space is 1-dimensional, so is $E.$

Dijkstra-Zhang \cite{DZ}  showed  that the space of Lelek fans, 
endowed with the Vietoris topology, in the Cantor fan is homeomorphic to the separable Hilbert space.

Here we introduce some  notation that we will need later on. By $v$ we denote the {\em top} $v=(0,0)/\!\!\sim$ of the Cantor fan. 
For a point $x\in F$, let $[v,x]$ denote the closed line 
with endpoints $v$ and $x$. If $x$ is in the line segment $[v,y],$ we denote by $[x,y]$ the line segment $([v,y]\setminus [v,x])\cup\{x\}.$ 
Points in $F$ will be denoted by $(c,y)$, where $c\in C$ and $y\in [0,1]$.
Let $\pi_1\colon F\setminus\{v\}\to C$, $\pi_1(c,x)=c$, 
and $\pi_2\colon F\to [0,1]$, $\pi_2(c,x)=x$,
 be  projections. 
Let $E$ be the set of endpoints of the Lelek fan $L$, and let
$H(L)$ be the group of all homeomorphisms of the Lelek fan.

\subsection{Projective Fra\"{i}ss\'{e} limits}\label{frase}

 Given a language $\mathcal{L}$ that consists of relation symbols $r_i$,  with arity $m_i$,  $i\in I$, and function symbols $f_j$, 
  with arity $n_j$, $j\in J$,
a \emph{topological $\mathcal{L}$-structure} is a compact Hausdorff zero-dimensional second-countable space $A$ equipped with
closed relations $r_i^A\subseteq A^{m_i}$ and continuous functions $f_j^A\colon A^{n_j}\to A$, $i\in I, j\in J$.
A~continuous surjection $\phi\colon B\to A$ is an
 \emph{epimorphism} if it preserves the structure, that is, for a function symbol $f$ in $\mathcal{L}$ of arity $n$ and $x_1,\ldots,x_n\in B$ we require:
\[
 f^A(\phi(x_1),\ldots,\phi(x_n))=\phi(f^B(x_1,\ldots,x_n));
\]
and for a relation symbol $r$ in $\mathcal{L}$ of arity $m$ and $x_1,\ldots,x_m\in A$  we require:
\begin{equation*}
\begin{split}
&  r^A(x_1,\ldots,x_m) \\ 
&\iff \exists y_1,\ldots,y_m\in B\left(\phi(y_1)=x_1,\ldots,\phi(y_m)=x_m, \mbox{ and } r^B(y_1,\ldots,y_m)\right).
\end{split}
\end{equation*}
By an \emph{isomorphism}  we mean a bijective epimorphism.

For the rest of this section fix a language $\mathcal{L}$.
Let $\mathcal{G}$ be a family of finite topological $\mathcal{L}$-structures. We say that $\mathcal{G}$ is a \emph{ projective Fra\"{i}ss\'{e} family} if it is countable and the following two conditions hold:

(JPP) (the joint projection property) for any $A,B\in\mathcal{G}$ there are $C\in \mathcal{G}$ and epimorphisms from $C$ onto $A$ and from $C$ onto $B$;

(AP) (the amalgamation property) for $A,B_1,B_2\in\mathcal{G}$ and any epimorphisms $\phi_1\colon B_1\to A$ and $\phi_2\colon B_2\to A$, there exist $C\in\mathcal{G}$,
 $\phi_3\colon C\to B_1$, and $\phi_4\colon C\to B_2$ such that $\phi_1\circ \phi_3=\phi_2\circ \phi_4$.

A topological $\mathcal{L}$-structure $\mathbb{G}$ is a \emph{ projective Fra\"{i}ss\'{e} limit } of $\mathcal{G}$ if the following three conditions hold:

(L1) (the projective universality) for any $A\in\mathcal{G}$ there is an epimorphism from $\mathbb{G}$ onto $A$;

(L2) for any finite discrete topological space $X$ and any continuous function
 $f\colon \mathbb{G} \to X$ there are $A\in\mathcal{G}$, an epimorphism $\phi\colon \mathbb{G}\to A$, and a function
$f_0\colon A\to X$ such that $f = f_0\circ \phi$;

(L3) (the projective ultrahomogeneity) for any $A\in \mathcal{G}$ and any epimorphisms $\phi_1\colon \mathbb{G}\to A$ 
and $\phi_2\colon \mathbb{G}\to A$
there exists an isomorphism $\psi\colon \mathbb{G}\to \mathbb{G}$ such that $\phi_2=\phi_1\circ \psi$.

We will often use Remark \ref{coveri}. It follows immediately from (L2).
\begin{remark}\label{coveri}
Let $\mathbb{G}$ be the projective Fra\"{i}ss\'{e} limit of $\mathcal{G}$. Then every finite open cover 
can be {\em refined by an epimorphism, i.e. for every open cover
 $\mathcal{U}$
of $\mathbb{G}$  
there is an epimorphism  $\phi\colon\mathbb{G}\to A$,  for some  
  $A\in\mathcal{G}$, such that for every $a\in A$, $\phi^{-1}(a)$ is contained in some open set in $\mathcal{U}$.}
\end{remark}

\begin{remark}
In the projective Fra\"{i}ss\'{e} theory, a projective Fra\"{i}ss\'{e} family has properties dual to the joint embedding property and to the amalgamation property
from the (injective)  Fra\"{i}ss\'{e} theory.
We do not have a condition that corresponds to the hereditary property. Nevertheless, we can think of (L2) as a dualization of a
``cofinal hereditary property'': if   $\mathbb{K}$ is the Fra\"{i}ss\'{e} limit of a Fra\"{i}ss\'{e} family $\mathcal{K}$, then for any finite 
 $X\subseteq \mathbb{K}$ , there is $A\in\mathcal{K}$ with  $X\subseteq A\subseteq \mathbb{K}$.
\end{remark}

\begin{theorem}[Irwin-Solecki, \cite{IS}]\label{is}
 Let $\mathcal{G}$ be a projective Fra\"{i}ss\'{e} family of finite topological $\mathcal{L}$-structures. Then:
\begin{enumerate}
 \item there exists a projective Fra\"{i}ss\'{e} limit of $\mathcal{G}$;
\item any two projective Fra\"{i}ss\'{e} limits of $\mathcal{G}$ are isomorphic.
\end{enumerate}
\end{theorem}

We will frequently use the following property of the projective Fra\"{i}ss\'{e} limit,  called {\em{ the extension property}}.

\begin{proposition}\label{fraisse}

If $\mathbb{G}$ is the projective Fra\"{i}ss\'{e} limit of $\mathcal{G}$ the following condition  holds:
Given $A,B\in\mathcal{G}$ and epimorphisms $\phi_1\colon B\to A$ and $\phi_2\colon \mathbb{G}\to A$, there is  an epimorphism 
$\psi\colon \mathbb{G}\to B$ such that $\phi_2=\phi_1\circ \psi$.
\end{proposition}

\subsection{Summary of results}
In Section 2, 
we construct the Lelek fan $L$ as a natural quotient of the
projective Fra\"{i}ss\'{e} limit of a   family of finite ordered trees. In fact, we show that we can restrict our attention to a  subclass $\f$ of simple trees called 
fans. We then use this construction
to show projective universality and projective ultrahomogeneity of the Lelek fan in the family of all smooth fans (Theorem \ref{uph}).
In particular, we obtain that every smooth fan is a continuous image of the Lelek fan.

In Section 3, we prove that  the homeomorphism group of the Lelek fan, $H(L)$, 
satisfies the following properties.
\begin{enumerate}
\item The group $H(L)$ is totally disconnected (Proposition \ref{discon}).
\item The group $H(L)$ is generated by every neighbourhood of the identity  (Corollary \ref{epsgen_c}).
\item The group $H(L)$ has a dense conjugacy class (Theorem \ref{xyz}).
\item The group $H(L)$ is simple (Theorem \ref{simple}).

\end{enumerate}
To prove properties (2) and (3), we use our projective Fra\"{i}ss\'{e} limit construction. 
For a detailed discussion of motivation, connections to other known results, etc., of each of these four properties, we refer to
 Section 3.

Lewis-Zhou (\cite{LZ}, Question 5) asked whether every homeomorphism group of a continuum,
which is generated by every neighbourhood of the identity, has to be connected. As $H(L)$ satisfies properties (1) and (2) above,
the answer to this question is negative.

We were recently informed by  Megrelishvili that results in this
paper together with results due to Ben Yaacov and Tsankov in
\cite{BYT} (Corollary 4.10) give a positive answer to a question posed
by Glasner and
Megrelishvili as Question 6.14 in \cite{Me} and Question 10.5(1) in \cite{GM}:
 Is it true that there exists a non-trivial Polish group $G$ which is
 reflexively trivial but does not contain
 $H_+[0,1]$, the group of increasing homeomorphisms of [0,1]?
Indeed, $H(L)$ provides an example of such a group.
As properties (1) and (2) above hold for $H(L)$ and since $\aut$, where $\lel$ is the 
projective Fra\"{i}ss\'{e} limit of the family of finite rooted reflexive fans discussed below,
is an oligomorphic group (see \cite{BKn}), Corollary 4.10 from \cite{BYT} implies
that $H(L)$ is reflexively trivial. Since $H(L)$ is totally
disconnected, it does not contain $H_+[0,1]$.

\section{Lelek fan as a quotient of a projective Fra\"{i}ss\'{e} limit}

\subsection{Construction of the Lelek fan}
Let $T$ be a {\em finite tree}, that is, an undirected simple graph which is connected and has no cycles.
We will only consider rooted trees, i.e. trees with a distinguished element $r_T\in T$. On a rooted tree $T$ there is a natural partial order
$\leq_T$: for $t,s\in T$ we let $s\leq_T t$ if and only if $s$ belongs to the path connecting $t$ and the root.
We say that $t$ is a {\em successor} of $s$ if $s\leq_T t$, $t\neq s$.
It is an {\em immediate successor} if additionally there is no $p\in T$, $p\neq s,t$,  with $s\leq_T p\leq_T t$.
A {\em chain} is a rooted tree $T$ on which the order $\leq_T$ is linear.
A {\em branch } of a rooted tree $T$ is a maximal chain in  $(T,\leq_T)$. 
If $b$ is a branch in $T$, we will sometimes write $b=(b(0),\ldots,b(n))$, where $b(0)$ is the root of $T$, and
$b(i)$ is an immediate successor of $b(i-1)$, for every $i=1, 2, \ldots, n$. 
By $B(T)$ we denote the set of all branches of $T$.

Let $R$ be a binary relation symbol. Consider the language $\mathcal{L}=\{R\}$.
For $s,t\in T$ we let $R^T(s,t)$ if and only if $s=t$ or 
 $t$ is an immediate successor of $s$.
 Let $\mathcal{T}$ be the family of all finite rooted trees, viewed as topological $\mathcal{L}$-structures, equipped with the discrete topology.

A function $\phi\colon (S,R^S)\to (T,R^T)$ is a {\em homomorphism} if for every $s_1,s_2\in S$,  whenever
$R^S(s_1,s_2)$ then $R^T(\phi(s_1),\phi(s_2))$.

\begin{remark}\label{epi}
Notice that $\phi\colon (S,R^S)\to (T,R^T)$ is an epimorphism if and only if it is a surjective homomorphism.
\end{remark}

Let  $\f$ be the family of {\em finite rooted reflexive fans}, that is, the family
that consists of rooted
 trees $T\in\mathcal{T}$ such that for every $s,t\in T$ which are incomparable in $\leq_T$, if 
$p\neq s,t$ is such that $R^T(p,s)$ and $R^T(p,t)$, then $p$ is the root of $T$, and moreover all branches of $T$ have the same length.

\begin{remark}\label{fan}
The family $\f$ is {\em coinitial} in $\mathcal{T}$, that is,
for every  $T\in \mathcal{T}$ there are  $S\in\f$ and an epimorphism $\phi\colon S\to T$.
\end{remark}

\begin{proposition}\label{ap}
The family $\mathcal{T}$ is a projective Fra\"{i}ss\'{e} family.
 \end{proposition}
 
 \begin{proof}
JPP: Take  $S_1, S_2\in \mathcal{T}$. Then the tree $T$  equal to the disjoint union of $S_1$ and $S_2$ with their roots identified, together with the natural projections from $T$ onto $S_1$ and from $T$ onto $S_2$ witness the JPP. 
 
AP: Take  $P,Q,S\in\mathcal{T}$ together with epimorphisms $\phi_1\colon Q\to P$ and $\phi_2\colon S\to P$. Without loss of generality, as $\f$ is coinitial in $\mathcal{T}$, $Q$ and $S$ are in 
 $\f$.

Let $b$ be a branch   in $Q,$ and let $a=\phi_1(b)$. Note that $a$ is an initial segment of a branch of $P$. Consider any  branch $c$ in $S$ such that $a\subseteq\phi_2(c)$.
Take a chain $d_b$ and $R$-preserving maps $\psi_1$ and  $\psi_2$ defined on $d_b$ (we do not require them to be surjective) such that $\psi_1(d_b)=b$,
$\psi_2(d_b)\subseteq c$,
 and for every $t\in d_b$, 
$\phi_1\circ \psi_1(t) =\phi_2\circ \psi_2(t) $.

We get $d_b$ for every branch $b$ in $Q$ and we  get $d_b$ for every branch $b$ in $S$. 
Without loss of generality, all chains $d_b$ are of the same length.
Let $T\in\f$ be 
the disjoint union of chains $d_b$ with their roots identified for $b$ a branch in $Q$ or in $S$.
Functions $\psi_1$ and $\psi_2$ are well defined on $T$, $\psi_1$ is onto $Q$, $\psi_2$ is onto $S$, and  
$\phi_1\circ \psi_1 =\phi_2\circ \psi_2$.

\end{proof} 

By Theorem \ref{is}, there exists a unique Fra\"{i}ss\'{e} limit of $\mathcal{T}$, which we denote by $\lel=(\lel, R^{\lel})$.

\noindent The following remark justifies that we can work only  with the family $\f$.

\begin{remark}
From Remark \ref{fan} and Proposition \ref{ap}, it  follows that $\f$ is a projective Fra\"{i}ss\'{e} family and by Theorem \ref{is}, the projective Fra\"{i}ss\'{e} limit of $\f$ is isomorphic to the one of $\mathcal{T}$.
 \end{remark}
 

For a topological $\mathcal{L}$-structure $\x,$ we define $R^{\x}_S$ to be the {\em symmetrization} of $R^{\x},$
that is,
$ R_S^{\x}(s,t)$ if and only if $ R^{\x}(s,t)$ or $  R^{\x}(t,s)$, for every  $s,t\in\x$.

\begin{theorem}
The relation $ R_S^{\lel}$ is an equivalence relation which has only one and two element equivalence classes.
\end{theorem}

\begin{proof}
To show that $ R_S^{\lel}$ is reflexive, take $x\in\lel$. From  (L2) in the definition of the projective Fra\"{i}ss\'{e} limit it follows that
for every clopen $U\subseteq \lel$ such that $x\in U$ there is $T\in \f$ and an epimorphism $\phi\colon \lel\to T$ refining the partition
$\{U, \lel\setminus U\}.$ By the definition of an epimorphism, there are $x_U,y_U\in U$ such that $R^{\lel}(x_U, y_U)$. Since $R^{\lel}$  is closed in $\lel\times \lel$,
it follows that $R^\lel(x,x)$, and therefore $R_S^{\lel}(x,x)$.

Clearly,  $ R_S^{\lel}$ is symmetric.

To finish the proof of the theorem, it suffices to show that
for every $p,q,r$, pairwise different, 
 we cannot have both  
$ R_S^{\lel}(p,q)$ and $ R_S^{\lel}(p,r)$. 
Suppose the opposite.
Since each member of $\f$ is a tree, it cannot happen that  $R^{\lel}(q,p)$ and $R^{\lel}(r,p)$. Therefore either we have
$R^{\lel}(p,q)$ and $R^{\lel}(p,r)$  or we have $R^{\lel}(q,p)$ and $R^{\lel}(p,r)$.
Consider a clopen partition $P$ of $\lel$ such that $p,q,r$ are in different clopens of $P$.
Using (L2) in the definition of the projective Fra\"{i}ss\'{e} limit, take $T\in \f$ and an epimorphism $\psi_1\colon \lel\to T$ refining $P$.
We have  that $p'=\psi_1(p)$, $q'=\psi_1(q)$ and $r'=\psi_1(r)$ are pairwise different and that
$R^T(p',q')$ (or $R^T(q',p')$, respectively) and $R^T(p',r')$.  Take $S$ which is equal to $T$ with $p'$ ``doubled'', i.e.
 let $S=T\cup\{ \bar{p}'\}$, $R^S(\bar{p}',\bar{p}')$,
 $R^S(p',\bar{p}'), R^S(\bar{p}',r')$, and for $x,y\in T$, $(x,y)\neq (p',r'),$ we let $R^S(x,y) $ if and only if
$R^T(x,y)$.
 Then $\phi\colon S\to T$ that sends $\bar{p}' $ to $p'$, and other points to themselves, is an epimorphism.
 Using the extension property,  we get an epimorphism $\psi_2\colon \lel\to S$ such
that $\psi_1=\phi\circ \psi_2$.
Then either $\psi_2(p)=p'$ or $\psi_2(p)=\bar{p}'$.
Either option leads to a contradiction.



\end{proof}

Take the quotient $\lel/R_S^{\lel}$ and denote it by $L$. 
Let $\pi\colon \lel\to L$ be the quotient map.

\begin{theorem}\label{Tlelek}
The space $L$ is the Lelek fan.
\end{theorem}

In order to prove Theorem \ref{Tlelek}, we will show that $L$ is a continuum, it embeds into the Cantor fan $F$ and has a dense set of endpoints.

\begin{lemma}
The space $L$ is Hausdorff, compact,  second-countable, and connected.
\end{lemma}
\begin{proof}
Since $\lel$ and $R_S^{\lel}$ are compact and $\pi$ is continuous, it  follows
that $L$ is Hausdorff, compact, and second-countable, since $\lel$ is such.

Suppose, towards a contradiction, that $L$ is not connected. Let $U$ be a clopen non-empty subset of $L$ such that 
$L\setminus U$ is also non-empty. Let $V=\pi^{-1}(U)$. Let $T\in \f$ and let $\phi\colon \lel\to T$ be an epimorphism
refining the partition $\{V,\lel\setminus V\}$. It follows that there are $x\in V$ and $y\in \lel\setminus V$ such that 
$R^{\lel}(x,y)$. Since $\pi(x)=\pi(y)$, we get a contradiction.

\end{proof}


\smallskip

We call a sequence $(T_n,f_n)$ 
an \emph{inverse sequence} if $T_n\in \f$ and $f_n\colon T_{n+1}\to T_n$ are epimorphisms for every $n$. We will denote by $f^n_m$ the composition $f_{m}\circ\ldots\circ f_{n-1}$, whenever  $m<n$,
and $f^m_m=\text{Id}_{T_m}$. If $\mathbb{T}$ is the inverse limit of $(T_n,f_n),$ then there is a sequence of epimorphisms $f^\infty_n\colon \mathbb{T} \to T_n$  
such that  $f^n_m\circ f^\infty_n= f^\infty_m$. If $(T_n, f_n)$ and $(S_n,g_n)$ are two inverse sequences with inverse limits $\mathbb{T}$ and $\mathbb{S}$ respectively, and for every $n$ there is an injective homomorphism $\iota_n\colon T_n\to S_n$ such that $\iota_n\circ f_n=g_n\circ \iota_{n+1},$ then there is a continuous homomorphic embedding  $\iota_{\infty}\colon \mathbb{T}\to\mathbb{S}$ satisfying $\iota_n\circ f^{\infty}_n=g^{\infty}_n\circ\iota_{\infty}.$ 

Following the proof of Theorem 2.4 in \cite{IS}, we can write $\lel$ as the inverse limit of an inverse sequence $(T_n,f_n)$  satisfying the following properties.

\begin{enumerate}

\item For any $T\in\f$ there is an $n$ and an epimorphism from $T_n$ onto $T$.

\item For any $m$, any $S,T\in\f$, and  epimorphisms $\phi_1\colon T_m\to T$ and
$\phi_2\colon S\to T$, there exists $m < n$ and an epimorphism $\phi_3\colon T_n\to S$ such that
$\phi_1\circ f_m^n= \phi_2\circ\phi_3$.

\end{enumerate}

For $T\in\f$ let as before $B(T)$ denote the set of branches of $T$.

By passing to a subsequence, we can assume that  $(T_n,f_n)$ moreover satisfies: 

\begin{enumerate}
\item[(3)] For every $b\in B(T_{n+1})$ and $x\in b$, there is $x'\in b$, $x'\neq x$, such that $f_n(x)=f_n(x')$.

\item[(4)] For every  $b\in B(T_{n})$ there are $b_1\neq b_2\in B(T_{n+1})$ such that $f_n(b_1)=f_n(b_2)=b$.
 \end{enumerate}

Any sequence $(T_n,f_n)$ that satisfies properties (1), (2), (3), and (4) above will be called a {\em Fra\"{i}ss\'{e} sequence}.


Our  goal now is to show the following proposition.

 \begin{proposition}\label{embed}
The continuum  $L$ can be embedded into the Cantor fan~$F$.
 \end{proposition}

Let $\ian$ be the inverse limit of any inverse sequence $(I_n,h_n)$, where $I_n$ is a finite chain and 
$h_n\colon I_{n+1}\to I_{n}$ is an epimorphism such that
for every $x\in I_{n+1}$, there is $x'\in I_{n+1}$, $x'\neq x$, with $h_n(x)=h_n(x').$ Then it is easily seen that $R^{\ian}_S$ has only one and two element equivalence classes and $\ian/R^{\ian}_S$ is homeomorphic to the unit interval $[0,1].$

An inverse limit of an inverse sequence $(C_n,e_n),$ where $C_n$ is a finite set and $e_n\colon C_{n+1}\to C_n$ is a surjection such that for every $x\in C_{n+1}$ there is $x'\in C_{n+1},$ $x'\neq x$ with $e_n(x)=e_n(x'),$ is clearly homeomorphic to the Cantor set.

It follows that if $\fan$ is the inverse limit of an inverse sequence $(S_n,g_n)$ 
  satisfying the  conditions (3) and (4) in the definition of a Fra\"{i}ss\'{e} sequence and a condition (5) below,  
	then $R^{\fan}_S$ has only one and two element equivalence classes and $\fan/R^{\fan}_S$ is homeomorphic to the Cantor fan $F.$ 

\begin{itemize}

\item[(5)]  For every $b\in B(S_{n})$ and  $b'\in B(S_{n+1})$ such that 
$g_{n}(b')\subseteq b$, we have $g_{n}(b')= b$.
\end{itemize}

We will find an injective,  continuous homomorphism  $h\colon \lel\to \fan$, 
which will induce a topological embedding 
from $L$ into $F$.

\begin{lemma}\label{three}
Suppose that 
 $({T}_n,{f}_n)$ is a Fra\"{i}ss\'{e} sequence. 
Then there is  an inverse sequence  $(S_n,g_n)$ 
satisfying (3),(4) and (5) above 
such that $T_n\subset S_n$ and $g_n \restriction {T}_{n+1}=f_n$ for every $n.$

In particular,  the inclusions induce a continuous injective homomorphism $h$ from the inverse limit $\lel$ of  $({T}_n,{f}_n)$ to the inverse limit $\fan$ of $(S_n,g_n).$
\end{lemma}

\begin{proof}

Let ${S}_0=T_0.$ Suppose that ${S}_k$ and ${g}_{k-1}$ 
have been constructed for $k\leq n.$ We will construct $S_{n+1}$ from $T_{n+1}$ by adding nodes and branches and we will define $g_n:S_{n+1}\to S_n$ 
to be equal to $f_n$ when restricted to $T_{n+1}.$ For every $b\in B(T_{n+1}),$ 
let  $b'\in B(S_n)$  be the branch 
such that  $ f_n(b)\subset b'.$ Let $e,e'$ denote the endpoints of $b,b'$ respectively, and let  $m_{b'}= f_n(e).$ 
For every $x\in b'$ such that $m_{b'}<_{S_n} x$, we will put two points $x_1\neq x_2$ into $S_{n+1}$ and set $R^{S_{n+1}}(x_1,x_2),$ $R^{S_{n+1}}(x_i,x_i)$  and $g_{n}(x_i)=x$ for $i=1,2$. 
If $R^{S_n}(m_{b'},x),$ then $R^{S_{n+1}}(e,x_1).$ If $m_{b'}<_{S_n} x <_{S_n} y$ and $R^{S_n}(x,y),$ then $R^{S'_{n+1}}(x_2,y_1).$ Finally, 
for every branch $c$ in $S_n\setminus T_n\cup \{r_{T_n}\},$ 
we will add two branches $c_1,c_2$ to $S_{n+1}$ that map onto $c$ under $g_{n}$ and such that for every $x\in c$ there are $x'\neq x''\in c_i$ such that $g_n(x')=g_n(x'')=x$ for $i=1,2.$

\end{proof}

\begin{proof}[Proof of Proposition \ref{embed}]
The continuous injective homomorphism $h$ from Lemma \ref{three} induces a continuous embedding between the respective quotients $L$ and $F$.
\end{proof}

\bigskip

Finally, we show the density of endpoints of $L$.
Let $A$ be a topological $\mathcal{L}$-structure.
We say that $K\subseteq A$ is {\em R-connected} if for every two non-empty, disjoint clopen subsets $K_1,K_2$ in $K$  such that $K_1\cup K_2=K$,
there are $x\in K_1$ and $y\in K_2$ such that $R^A(x,y)$ or $R^A(y,x)$.
We again consider $\lel$ as an inverse limit of a Fra\"{i}ss\'{e} sequence  $(T_n,f_n)$.
Let $r_n=r_{T_n}$ denote the root of $T_n$, and $r=(r_n)$  the top of $\lel$. Recall that $\pi\colon\lel\to L$ is the quotient map.

\begin{proposition}
The set $E$ of all endpoints in $L$ is dense in $L$.
\end{proposition}

\begin{proof}
 Let $U\subseteq L$ be open and non-empty. 
We will find an endpoint in $U$. Let $V=\pi^{-1}(U)$. 
Take $n_1$ such that there is $e_{n_1}\in T_{n_1}$ with $(f^\infty_{n_1})^{-1}( e_{n_1})\subseteq V$.
Let $T\in\f$, $\psi_1\colon T\to T_{n_1}$, and $x\in T$ be such that $\psi_1(x)=e_{n_1}$ and $x$ is an endpoint of $T$ 
(i.e., 
for no $y\in T$, $y\neq x$, we have $R^T(x,y)$). 
Using that $(T_n,f_n)$ is a Fra\"{i}ss\'{e} sequence, find $n_2$ and $\psi_2\colon  T_{n_2}\to T$ 
with $f^{n_2}_{n_1}=\psi_1\circ \psi_2$. 
Pick any endpoint $e_{n_2}\in T_{n_2}$ in the preimage of $x$ by $\psi_2$. 
For $n>n_2$ inductively pick an endpoint $e_n$ in $T_n$ such that $f^n_{n-1}(e_n)=e_{n-1}$ and 
for $n<n_2$ let $e_n=f^{n_2}_n(e_{n_2})=e_n$. 
Then $e=(e_n)\in V,$ and therefore $\pi(e)\in U$.
Moreover, $e$ is not the root of $\lel$ as $e_{n_2}$ is not the root of $T_{n_2}$.
By the property (2) in the definition of a Fra\"{i}ss\'{e} sequence, $\pi^{-1}(\pi(r))=\{r\}$ for $r$ the root of $\lel$. Consequently, $\pi(e)\neq\pi(r)$.

We show that $\pi(e)\in E$.  
 Let $i\colon [0,1]\to L$ be a homeomorphic embedding  such that $\pi(e)\in i(I).$
Suppose towards a contradiction that $\pi(e)\neq i(0)$ and $\pi(e)\neq i(1)$.  Without loss of generality,  $\pi(r)\notin i(I).$
Denote 
 $X=\pi^{-1}(i([i^{-1}(\pi(e)),1]))$,  $Y=\pi^{-1}(i([0,i^{-1}(\pi(e))]))$, and $Z=\pi^{-1}(i([0,1]))$.
All three sets $X, Y, Z$ are compact, $R$-connected in $\lel$ and $e\in X\cap Y$ 
Let  $X_n=f^\infty_n(X)$, $Y_n=f^\infty_n(Y)$, and $Z_n=f^\infty_n(Z)$.
All  sets $X_n,Y_n, Z_n$ are $R$-connected in $T_n$.
Since  $\pi(r)\notin i(I),$ there is $N>n_2$ such that whenever $n>N$, $Z_n$ (and so $Y_n$ and 
$X_n$) is contained in a single branch of $T_n.$

Let $x=(x_n)\in X\setminus Y$ and let $y=(y_n)\in Y\setminus X$. 
We have that ${e}=(e_n)\in X\cap Y$.
Then  either for every $n>N$,  $r_n<_{T_n} x_n<_{T_n} y_n<_{T_n} {e}_n$, or  for every  $n>N$, $r_n<_{T_n} y_n<_{T_n} x_n<_{T_n} {e}_n$.
Without loss of generality, we may assume that the former holds.
Since for every $n$, we have $x_n,e_n\in X_n$, $R$-connectivity of each $X_n$ implies $y_n\in X_n$ for $n>N.$ Therefore $y\in X$, which is a contradiction.

\end{proof}


\subsection{Properties of the Lelek fan: projective universality and projective ultrahomogeneity}\label{pu}

The main goal of this subsection is to prove Theorem \ref{uph}. This is an analog of Theorem 4.4. in \cite{IS}.

Let $\aut$ be the group of all automorphisms of $\lel$, that is, the group of all homeomorphisms of $\lel$ that preserve the relation $R^{\lel}$.
This is a topological  group when equipped with the compact-open topology inherited from $H(\lel)$,
 the group of all homeomorphisms of the Cantor set underlying the structure $\lel$.
Since $R^\lel$ is closed in $\lel\times \lel$, the  group $\aut$ is closed in $H(\lel)$. 

We will denote by $H(L)$ the group of homeomorphisms of the Lelek fan with the compact-open topology.

\begin{remark}\label{topo}
\begin{enumerate}
\item Every $h\in \aut$ induces a homeomorphism  $h^*\in H(L)$
satisfying $h^*\circ\pi(x)=\pi\circ h(x)$ for $x\in\lel$. 
The map $h\to h^*$ is injective and 
we will frequently identify $\aut$ with the corresponding subgroup $\{h^*\colon h\in \aut\}$ of $H(L)$. 
\item The group $\aut$ is nontrivial by projective ultrahomogeneity, which immediately implies that $H(L)$ is nontrivial.
\item The compact-open topology  on $\aut$ is finer than the topology on $\aut$ that is inherited from the compact-open 
topology on $H(L)$.
\end{enumerate}
\end{remark}

A continuum is {\em hereditarily unicoherent} if the intersection of any two subcontinua is connected.
A {\em dendroid} is a hereditarily unicoherent and arcwise connected continuum.
A point $x$ of a dendroid $X$ is a {\em ramification point} if there are $a,b,c\in X$ and arcs $ax, bx, cx$ such that $ax\cap bx=\{x\}$,
$bx\cap cx=\{x\}$, and $ax\cap cx=\{x\}$.
A {\em fan} is a dendroid that has exactly one ramification point, called the {\em top}.
A {\em smooth fan} $X$ is a fan such that   whenever $t_n\to t$, $t_n, t\in X$ then the sequence of
arcs $t_nw$ converges to the arc $tw$ (in the Hausdorff metric),  where $w$ is the top point of $X$.
Smooth fans  are exactly  fans that can be embedded into the Cantor fan
 (see \cite{CC}, Proposition 4). These are exactly non-degenarate subcontinua of the Cantor fan $F$
 that are not homeomorphic to the interval [0,1].

We say that a continuous surjection $f\colon L\to X$, where $X$ is a smooth fan, is {\em monotone on segments} if 
$f(v)=w$, where $v$ is the top of $L$ and $w$ is the top of $X$, and 
for every $x,y\in L$ such that $x\in[v,y]$, we have $f(x)\in [w,f(y)]$.

\begin{theorem}\label{uph}
\begin{enumerate}
\item Each smooth fan is a continuous image of the Lelek fan $L$ via a map that is monotone on segments.
\item Let $X$ be a smooth fan with a metric $d$.
If $f_1, f_2\colon L\to X$ are two continuous  surjections that are monotone on segments,  then for any $\epsilon>0$ there exists $h\in \aut$
such that for all $x\in L$, $d(f_1(x),f_2\circ h^*(x))<\epsilon$.
\end{enumerate}
\end{theorem}

In order to prove Theorem \ref{uph}, we will represent every smooth fan as a quotient of an inverse limit of elements in $\f$ and apply the following proposition by Irwin and Solecki.

\begin{proposition}[\cite{IS}, Proposition 2.6]\label{refine}
Let $\mathcal{G}$ be a projective Fra\"{i}ss\'{e} family of finite topological $\mathcal{L}$-structures and let $\mathbb{G}$ be 
its projective Fra\"{i}ss\'{e}  limit. Let $X$ be a topological $\mathcal{L}$-structure such that any open cover of $X$ is 
refined by an epimorphism onto a structure in $\mathcal{G}$. Then there is an epimorphism from $\mathbb{G}$ onto $X$.
\end{proposition} 

Moreover, we will show that the epimorphism between the limits as in Proposition \ref{refine} induces a continuous surjection monotone on segments between the respective continua.


\begin{lemma}\label{cover}
Let $\epsilon>0$. Let $X$ be a smooth fan with the top $w$. Then there is $A\in\f$ and  an open cover $(U_a)_{a\in A}$ of $X$ 
such that  
\begin{itemize}
\item[(C1)] for each $a\in A$, $\text{diam}(U_a)<\epsilon$, 
\item[(C2)] for each $a,a'\in A$, if $U_a\cap U_{a'}\neq\emptyset$ then $R_S^A(a,a') $,
\item[(C3)] for each $x,y\in X$ with $y\in [w,x]$, if $y\in U_a$ and $x\in U_b$, 
but $\{x,y\}\not\subset U_a\cap U_b$, unless $a=b$,
then $a\leq_A b$, 
\item[(C4)] for every $a\in A$ there is $x\in X$ such that $x\in U_a\setminus(\bigcup \{U_{a'}\colon a'\in A, a'\neq a\}).$
\end{itemize}
\end{lemma}

\begin{remark}
 Note that (C3) implies that  $w\in U_{r_A}$, where $r_A$ is the root of $A$, and that if $a,a'\in A$ satisfy $R_S^A(a,a')$ then $U_a\cap U_{a'}\neq\emptyset$.
\end{remark}

\begin{proof}[Proof of Lemma \ref{cover}]
We first show that the lemma holds for the Cantor fan $F$.
 Let $\{O_1,O_2,\ldots, O_n\}$ be an open $\frac{\epsilon}{2}$-cover of the unit interval $I=[0,1]$ such that for every $i,j$,  $O_i\cap O_j\neq\emptyset$ if and only if $|i-j|\leq 1$ and for  $x\in O_i\setminus O_j$ and $y\in O_j$ with $i<j$ we have $x<y.$ 
Moreover we require $O_i\setminus O_j\neq \emptyset$ whenever $i\neq j.$ 
 Let $\{V_1,V_2,\ldots,V_m\}$ be a clopen $\frac{\epsilon}{2}$-cover of the Cantor set $C$. Then $\{O_i\times V_j\colon i=1,2,\ldots,n, j=1,2,\ldots,m\}$ is an open $\epsilon$-cover of $C\times I.$ Let $O\subseteq F$ be an open neighbourhood of the top $w$  of $F$ of the form $O=\bigcup_{j=1}^m O_1\times V_j/\!\!\sim,$ where $(a,b)\sim (c,d)$ if and only if either $a=c$ and $b=d$ or $b=d=0.$ The desired cover is then $\mathcal{V}=\{O\}\cup \{O_i\times V_j\colon i=2,\ldots,n, j=1,\ldots,m\}$ with $A=\{r\}\cup\{(i,j)\colon i=2,\ldots,n,j=1,\ldots,m\},$ where  $R^A(r,(i,j))$ if and only if $j=2$ and for $(i,j),(i',j')\in \{2\ldots,n\}\times\{1,\ldots,m\},$ $R^A((i,j),(i',j'))$ if and only if $i=i'$ and $0\leq j'-j\leq 1.$

If $X$ is a smooth fan, we think of $X$ as embedded in $F$ and obtain the desired cover as $\{V\cap X\colon V\in \mathcal{V}\},$ and the structure $A$ from the one defined for $F.$ 
We can further arrange that all branches of $A$ have the same length. 
\end{proof}

\begin{proposition}\label{representation}
For every smooth fan $X,$ there exists a topological $\mathcal{L}$-structure $(\mathbb{X},R^{\mathbb{X}})$ such that $R^{\mathbb{X}}_S$ has one and two element equivalence classes and  $X$ is homeomorphic to $\mathbb{X}/R^{\mathbb{X}}_S.$ 
Moreover, 
every finite open cover of $\mathbb{X}$ can be refined by an epimorphism onto a fan in $\f.$
\end{proposition}
\begin{proof}
Let $X$ be a smooth fan viewed as a subfan of the Cantor fan $F$. 
While proving Proposition  \ref{embed} we already described how to obtain the Cantor fan as a quotient  of a topological $\mathcal{L}$-structure.

Let $C$ be the Cantor set viewed as the middle third Cantor set.
Each point of $C$ can be expanded in a ternary sequence $0.a_1a_2a_3\ldots$, where for each $i$, $a_i\in\{0,2\}$.
Similarly, each point of $[0,1]$ can be expanded in a binary  sequence 
$0.a'_1a'_2a'_3\ldots$, where for each $i$, $a'_i\in\{0,1\}$.
Let $f\colon C\to [0,1]$ be given by $f(0.a_1a_2a_3\ldots)=0.a'_1a'_2a'_3\ldots$,
where $a'_n=0$, when $a_n=0$, and $a'_n=1$, when $a_n=2$.
Consider $\text{Id}\times f/\!\!\sim \colon C\times C/\!\!\!\sim\,\,\to F$, where 
$(a,b)\sim (c,d)$ if and only if $a=c$ and $b=d$, or $b=d=0$.

Let $\x=(\text{Id}\times f/\!\!\sim)^{-1}(X)$. Set 
$R^\x((a,b),(c,d)) $ if and only if $a=c$ and $b=d$, or $a=c$ and $(b,d)$ is an interval removed from $[0,1]$ in the construction of $C$.
Then $\x=(\x,R^\x)$ is a topological $\mathcal{L}$-structure. Observe that $\x/R_S^\x=X$. 

{To prove the ``moreover'' part, observe that the following claim is true and it can be proved analogously to Lemma \ref{cover}.
\begin{claim}\label{cover2}
For every $\epsilon>0$ there exist $A\in\f$ and an epimorphism $\phi\colon \mathbb{X}\to A$ such that for each $a\in A$, $\text{diam}(\phi^{-1}(a))<\epsilon$,
where the diameter is taken with respect to some fixed compatible metric on $\mathbb{X}$.
\end{claim}
Now, if we have an open cover of $\mathbb{X}$, since $\mathbb{X}$ is compact, by the Lebesgue covering lemma we can find an $\epsilon>0$ such that 
the epimorphism guaranteed by Claim \ref{cover2} is as required.}
\end{proof}

\begin{proof}[Proof of Theorem \ref{uph}]
(1) 
Let $X$ be a smooth fan and let $\x$ be as in Proposition \ref{representation}. 
By Proposition \ref{refine}, there is an epimorphism  $f\colon\lel\to \x$.
This epimorphism induces a continuous surjection  $\bar{f}$ from $L=\lel/R_S^{\lel}$ onto $X=\x/R_S^\x$. It remains to show that $\bar{f}$ is monotone on segments. 
Let $\pi\colon\lel\to L$  be the quotient map. 
Clearly, $\bar{f}(v)=w,$ where $v$ and $w$ are the tops of $L$ and $X$ respectively.
Let  $x,y\in \lel$ be such that 
$\pi(x)\in [v,\pi(y)].$ 
We  show that $\bar{f}(\pi(x))\in [w,\bar{f}(\pi(y))].$   Let $T\in \f$ and let $\phi\colon\x\to T$ be an epimorphism 
that separates $f(x)$ and $f(y)$ and such that if $[w,\bar{f}(\pi(x))]\cap [w,\bar{f}(\pi(y))]=\{w\},$
then $\phi\circ f(x)$ and $\phi\circ f(y)$ are in different branches of $T$.
Since $\pi(x)\in [v,\pi(y)]$  and 
$\phi\circ f$ is an epimorphism, we have 
$\phi\circ f(x)\leq_T \phi\circ f(y).$ 
 Now, since $\phi$ is an epimorphism,
 we conclude that
 $\bar{f}(\pi(x))\in[w,\bar{f}(\pi(y))].$

(2)
Take 
$A\in \f$ and an open $\epsilon$-cover $\{U_a\}_{a\in A}$ of $X$ 
as in  Lemma \ref{cover}. 
Using the Lebesgue covering lemma, find $\delta$  such that for every  $M\subseteq X$ with $\text{diam}(M)<\delta$ there exists $a\in  A$ such that $M\subseteq U_a.$ 
Since $f_1\circ \pi$ and $f_2\circ \pi$ are uniformly continuous on $\lel$, there is $B\in \f$
and  epimorphisms $\phi_i\colon \lel\to B$, $i=1,2$ such that for $b\in B$,
$ \text{diam}(f_i\pi\phi_i^{-1}(b))<\delta$. Let $A$ and $\psi_i\colon B\to A$ be defined as follows:
$\psi_i(b) =a$ if and only if $f_i\pi  \phi_i^{-1}(b)\subseteq U_a$ and whenever 
 $f_i\pi  \phi_i^{-1}(b)\subseteq U_{a'}$, then $R^A(a',a)$.
 
 We show  that $\psi_i$ is an epimorphism for $i=1,2.$ Firstly, $\psi_i$ is onto. That follows from the fact that $\{U_a\colon a\in A\}$  and $\{f_i\pi\phi_i^{-1}(b)\colon b\in B_i\}$ are covers of $X$, and  from (C4).

 Secondly, let $b,b'\in B$ be such that $R^{B}(b,b')$. Since $\phi_i$ is an epimorphism, $\pi\phi_i^{-1}(b)\cap\pi\phi_i^{-1}(b')\neq\emptyset,$ and consequently
  $f_i\pi\phi_i^{-1}(b)\cap f_i\pi\phi_i^{-1}(b')\neq\emptyset$, and therefore  $U_{\psi_i(b)}\cap U_{\psi_i(b')}\neq\emptyset$. 
By (C2),   $R^A(\psi_i(b),\psi_i(b'))$  or
 $R^A(\psi_i(b'),\psi_i(b))$. We will show that only the former is possible whenever $\psi_i(b)\neq\psi_i(b').$ 
 
 
 
 Suppose on the contrary 
 that $R^A(\psi_i(b'),\psi_i(b)).$ 
By the definition of $\psi_i$, there exists 
 $x_i\in f_i\pi\phi_i^{-1}(b')\setminus U_{\psi_i(b)}\subseteq U_{\psi_i(b')}\setminus U_{\psi_i(b)}.$ Let $s_i,s_i'\in\mathbb{L}$ be such that
 $s_i\in[r,s'_i]$, where $r$ is the top of $\lel$,
$\phi_i(s_i)=b,$ $\phi_i(s_i')=b'$ and $f_i\pi(s_i')=x_i.$ It follows that $\pi(s_i)\in[v,\pi(s_i')],$ and since $f_i$ is monotone on segments, also $f_i \pi(s_i)\in[w,f_i \pi(s_i')=x_i].$ This however contradicts  (C3) as $f_i \pi(s_i)\in U_{\psi_i(b)}$ and $f_i \pi(s_i')\in U_{\psi_i(b')}$.
 
 

 We proved that $\psi_i$'s are surjective homomorphisms. By Remark \ref{epi}, they are automatically epimorphisms.
 
 
 

Finally, by (L3), there exists $h\in\aut$ such that $\psi_1\circ\phi_1=\psi_2\circ\phi_2\circ h$.
This gives that for each $y\in \lel$, there is $a\in A$ such that $f_1\circ\pi(y), f_2\circ \pi\circ h(y)\in U_a$.
Hence for all $x\in L$, $d(f_1(x),f_2\circ h^*(x))<\epsilon$.

\end{proof}

\begin{corollary}\label{dense}\label{dense}
The group $\aut$ is dense in $H(L)$.
\end{corollary}
\begin{proof}
In (2) of Theorem \ref{uph} take $X=L$, an arbitrary $f_1\in H(L)$, and $f_2={\rm Id}$.
\end{proof}

A metric space $X$ is {\em uniformly pathwise connected } if there exists a family $P$ of paths in $X$ such that 

\noindent (1) for $x,y\in X$ there is a path in $P$ joining $x$ and $y$, and 

\noindent (2) for every $\epsilon >0$ there is a positive integer $n$ such that each path in $P$ can be partitioned into $n$ pieces of diameter at most $\epsilon$. 

Kuperberg \cite{K} showed that continuous images of the Cantor fan  are precisely uniformly pathwise connected continua.

Since the Lelek fan is  clearly uniformly pathwise connected, it is 
a continuous image of the Cantor fan, and since
the Cantor fan is a continuous image of the Lelek fan (by the first part of Theorem \ref{uph}), we obtain the following corollary.

\begin{corollary}
Continuous images of the Lelek fan are  precisely uniformly pathwise connected continua.
\end{corollary}

\section{The homeomorphism group of the Lelek fan}

\subsection{Connectivity properties of $\mathbf{H(L)}$}

We show that $H(L)$ -- the homeomorphism group of the Lelek fan $L$ -- is totally disconnected (Proposition \ref{discon})
and is generated by every neighbourhood of the identity (Corollary \ref{epsgen_c}). 
A topological space $X$ is {\em totally disconnected } if for any $x,y\in X$ there is a clopen set $U\subseteq X$ such that
$x\in U$ and $y\in (X\setminus U)$. Note that this implies that every subspace of $X$ containing more than one element is  not connected
(the latter property is in literature  often used as a definition of being totally disconnected).

 We say that a metric group $(G,d)$ is {\em generated by every neighbourhood of the identity} if
for every $\epsilon>0$ and $h\in G$, there are homeomorphisms 
$ f_1,\ldots, f_n\in G$ such that for every $i$, $d(f_i,\text{Id})<\epsilon$, 
and $h=f_1\circ\ldots \circ f_n$.
The definition of being generated by every neighbourhood of the identity naturally extends to topological groups,
however, we will only need it in the context of metric groups.
Lewis  \cite{Le} showed that the homeomorphism group of the pseudo-arc is generated by every neighbourhood of the identity. However, it is not known whether that group is totally disconnected (see \cite{Le2}, Question 6.14).
There are examples of totally disconnected {\em Polish groups }
(separable and completely metrizable topological groups)
that are generated by every neighbourhood of the identity.
The first such example, solving Problem 160 in the Scottish Book (see \cite{M}), posed by  Mazur, asking whether a complete metric group that is generated by every neighbourhood of the identity must be connected, was given by Stevens \cite{S}; another example was presented by
 Hjorth \cite{H}. 
Groups constructed by Stevens and Hjorth are algebraically subgroups of the additive group of real numbers.
Our example is 
different from the two above. The group $H(L)$ is non-abelian, which is because $H(L)$ is non-trivial  (Remark \ref{topo}(2)) and 
it has non-trivial conjugacy classes (Theorem \ref{xyz}).
and it  is explicitly given as a homeomorphism group of a continuum.
Lewis-Zhou (\cite{LZ}, Question 5) asked whether every homeomorphism group of a continuum
that is generated by every neighbourhood of the identity has to be connected. Our example shows that the answer is negative.

As in Subsection \ref{notation}, let $C$ be the Cantor set,
$F$ the Cantor fan with the top point $v$,
 and let $\pi_1\colon F\setminus\{v\}\to C$ be the projection.

\begin{proposition}\label{discon}
The group $H(L)$ is totally disconnected.
\end{proposition}
\begin{proof}
Let $h_1\neq h_2\in H(L)$. We show that there is a clopen set $A$ in $H(L)$ such that $h_1\in A$ and $h_2\notin A$.
Since the set of endpoints $E$ is dense in $L$ and $h_1\neq h_2$, there is $e\in E$  such that $h_1(e)\neq h_2(e).$
Let $U_0$ be a clopen set in $C$ such that $\pi_1(h_1(e))\in U_0$ and $\pi_1(h_2(e))\notin U_0$ and
 let $U=\pi^{-1}(U_0)\cap E$. Then $U$ is a clopen set in $E$.
 Since $H(L)\to E, h\mapsto h(e)$ is continuous, $A=\{h\in H(L)\colon h(e)\in U\}$ is a clopen set in $H(L)$ such that $h_1\in A$ and $h_2\notin A$.
\end{proof}

Fix a compatible metric $d$ on $L$. Denote the corresponding supremum metric on $H(L)$ by $d_{\sup}$.
A homeomorphism $h\in H(L)$ is called an $\epsilon$-homeomorphism if  $d_{\sup}(h,\text{Id})<\epsilon$.

\begin{theorem}\label{epsgen}
For every $\epsilon>0$ and $h\in \aut$ there exist
$ f_1,\ldots, f_n\in \aut$ 
such that  $h=f_1\circ\ldots \circ f_n$ and $f^*_0,\ldots,  f^*_n$ are $\epsilon$-homeomorphisms.
\end{theorem}

\begin{proof}[Proof of Theorem \ref{epsgen}]
Let $\mathcal{B}_0$ be an open cover of $L$ 
that consists of sets of diameter $<\frac{\epsilon}{2}$.
Let $\mathcal{B}=\{\pi^{-1}(B)\colon B\in \mathcal{B}_0\}$ be an open cover of $\lel$, where $\pi\colon \lel\to L$ is the quotient map.
Let $S\in\f$ and $\alpha\colon \lel\to S$ be an epimorphism that refines $\mathcal{B}$.
Note that 
for $s,s'\in S$ 
\[(R^S(s,s')) \to 
(\text{diam}(\pi\circ\alpha^{-1}(s)\cup \pi\circ\alpha^{-1}(s')  )<\epsilon)  \tag{$\triangle$} \label{triangle}
\] 
since $\pi\circ\alpha^{-1}(s)\cap \pi\circ\alpha^{-1}(s') \neq\emptyset $.

Using the uniform continuity of $h$ and  the Lebesgue covering lemma, find a finite open cover $\mathcal{U}$ of $\lel$ refining $\mathcal{C}$
such that $h(\mathcal{U})=\{h(U)\colon U\in \mathcal{U}\}$ also refines $\mathcal{C}$.  Let $T\in \f$  and let $\gamma\colon\lel\to T$ be an epimorphim  refining $\mathcal{U}$. Then also $\mathcal{D}=\{\gamma^{-1}(t)\colon t\in T\}$ 
and $h(\mathcal{D})=\{h\circ \gamma^{-1}(t)\colon t\in T\}$
both refine $\mathcal{C}=\{\alpha^{-1}(s)\colon s\in S\}$. Denote by $\beta$ the surjection from $T$ onto $S$ such that $\alpha=\beta\circ \gamma$. We have that $\beta$ is an epimorphism as $\alpha$ and $\gamma$ are.

Let $\beta_0=\alpha\circ h\circ \gamma^{-1}$ and let $\gamma_0=\gamma\circ h^{-1}$.
Note that $\beta_0$ is an epimorphism and $\alpha=\beta_0\circ \gamma_0$.

 Without loss of generality, we can assume the following property.
\begin{align}
 & \text{For every branch in $S$ there are at least $k+1$ branches in $T$ }\tag{*} \\ & \text{that map onto the given branch under $\beta_0 $. }  \notag
\end{align}
If the original fan $T$ does not have this property, we take $T'$ and $\phi\colon T'\to T$ such that for every branch  $b$ in $T$
there are $k+1$ branches in $T'$ that are mapped by $\phi$ onto $b$. 
We apply the extension property to $\phi$ and $\gamma_0$ and get $\psi\colon\lel\to T'$ such that $\gamma_0=\phi\circ\psi$.
We replace $T$ by $T'$, $\gamma_0$ by $\psi$,  $\beta_0$ by $\beta_0\circ\phi$, $\gamma$ by  $\psi\circ h$,
and $\beta$ by  $\alpha\circ h^{-1}\circ \psi^{-1}$.

It is enough to construct epimorphisms  $\beta_1,\ldots, \beta_n=\beta\colon T\to S$, for some $n$,
such that for every $0\leq i<n$ and for every $t\in T$, $R^S(\beta_i(t),\beta_{i+1}(t))$ or $R^S(\beta_{i+1}(t),\beta_{i}(t))$.
Then using the extension property, we find $\gamma_1,\ldots, \gamma_n=\gamma$ such that $\alpha=\beta_i\circ\gamma_i$,
$i=1,2,\ldots, n$, while
projective ultrahomogeneity then provides us with $h=h_0,h_1,\ldots, h_{n-1},  h_n=\text{Id}\in\aut$ such that $\gamma=\gamma_i\circ h_i$.
For each automorphism $h_i,$ let $h_i^*$ denote the corresponding homeomorphism of $L$,  let
$ f_i=h_{i-1}^*\circ (h_i^*)^{-1}$,  $i=1,2,\ldots, n.$  Clearly, the  composition $f_1\circ\ldots \circ f_n$  
is equal to $h$, and each  $f_i$ is an
$\epsilon$-homeomorphism.
Indeed, for every $x\in\lel$ and $i=0,1,\ldots, n-1$, 
\[R^S(\alpha\circ h_i(x), \alpha\circ h_{i+1}(x)) \text{ or } 
R^S(\alpha\circ h_{i+1}(x),  \alpha\circ h_i(x)),
\]
 since 
 \[\alpha\circ h_i(x)=\beta_i\circ\gamma_i\circ h_i(x)=\beta_i\circ \gamma(x)
 \] 
 and 
 \[R^S(\beta_i(t),\beta_{i+1}(t)) \text{ or } R^S(\beta_{i+1}(t),\beta_{i}(t))
 \]
  for every $t\in T.$  
By (\ref{triangle}),
we get that for every $x\in\lel$,
\[{\rm diam}(\pi\circ\alpha^{-1}(\alpha\circ h_i(x))\cup \pi\circ\alpha^{-1}(\alpha\circ h_{i+1}(x)))<\epsilon,
\]
and therefore,
\[d_\text{sup}(h_i^*, h_{i+1}^* )=d_\text{sup}(h_i^*\circ (h_{i+1}^*)^{-1}, {\rm Id} )<\epsilon.
\] 


 Enumerate all branches in $S$ as $c_1,\ldots, c_k$  and all branches in $T$ as $d_1,\ldots,d_l$ in a way that the property below holds.
\begin{align}
& \text{ For every $1\leq i\leq k$, $\beta\restriction d_i$ is onto $c_i$.} \tag{**}
\end{align} 
Let $\beta_0(d_1)= (c(0), c(1),\ldots, c(m_1))\subseteq c$
 and $\beta(d_1)=(c_1(0),c_1(1),\ldots, c_1(m_2))\\ \subseteq c_1$.
We construct $\beta_1,\beta_2,\ldots,\beta_{n_1}$, for  $n_1=m_1+m_2$.
For $i=1,\ldots, m_1$, let \[
    \beta_i(t)= 
\begin{cases}
    c(m_1-i)  & \text{if }  t\in d_1 \text{ and }\beta_{i-1}(t)=c(m_1-i+1)\\
    \beta_{i-1}(t)              & \text{otherwise}.
\end{cases}
\]
For $i=1,\ldots, m_2$, let \[
    \beta_{m_1+i}(t)= 
\begin{cases}
    c_1(i)  & \text{if }  t\in d_1 \text{ and } \beta(t)\in\{c_1(i),\ldots, c_1(m_2)\}\\
    \beta_{m_1+i-1}(t)              & \text{otherwise}.
\end{cases}
\]
We continue in the same manner for $2,\ldots, l$ and construct  
 $\beta_{n_1+1},\ldots,\beta_{n_2}$,...,    $\beta_{n_{l-1}+1},\ldots,\beta_{n_l}$.
By $(*)$ and $(**)$, 
each $\beta_i$ is onto.
All $\beta_i$'s are  epimorphisms and they satisfy the required condition: for every $0\leq i< n$ and for every $t\in T$, $R^S(\beta_i(t),\beta_{i+1}(t))$ or $R^S(\beta_{i+1}(t),\beta_{i}(t))$.


\end{proof}  

Theorem \ref{epsgen} immediately yields the following corollaries. To obtain Corollary \ref{epsgen_c},
we also use Corollary \ref{dense}, which says that $\aut$ is dense in $H(L)$.

\begin{corollary}\label{epsgen_c}
The group $H(L)$ is generated by every neighbourhood of the identity.
\end{corollary}   

\begin{corollary}\label{open}
The group $H(L)$ has no proper open subgroup.
\end{corollary}

A Polish group is {\em non-archimedean} if it contains a basis of the identity that consists of open subgroups.
This class of groups is equal to the class of automorphism groups  of countable model-theoretic structures. 
\begin{corollary}\label{arch}
The group $H(L)$ is not a non-archimedean group.
\end{corollary}

 The following is a classical theorem about locally compact groups.

\begin{theorem}[van Dantzig, \cite{HR} (7.7)]\label{vd} 
A totally disconnected locally compact group admits a basis at the identity that consists of compact open subgroups.
\end{theorem}

Since $H(L)$ is totally disconnected (by Proposition \ref{discon}), the theorem above implies the following corollary.
\begin{corollary}\label{loc}
The group $H(L)$ is not locally compact.
\end{corollary}


\subsection{Conjugacy classes of $\mathbf{H(L)}$}

The main result of this subsection is the following theorem.
\begin{theorem}\label{xyz}
 The group of all homeomorphisms of the Lelek fan, $H(L)$, has a dense conjugacy class.
\end{theorem}

 This will follow from Theorem \ref{aut}.

\begin{theorem}\label{aut}
 The group of all automorphisms of $\lel$, $\aut$, has a dense conjugacy class.
\end{theorem}

Let us first see how Theorem \ref{aut} implies Theorem \ref{xyz}.
\begin{proof}[Proof of Theorem \ref{xyz}]
As noticed in Remark \ref{topo}, the group $\aut$ can be identified with a subgroup of $H(L)$ and its topology is finer than the one
inherited from $H(L)$. From Corollary \ref{dense}, $\aut$ is a dense subset of $H(L)$.  
From these observations, a set which is dense in $\aut$ is also dense in $H(L)$.
\end{proof}

To show Theorem \ref{aut}, we use the criterion stated in Proposition \ref{abc} below.
The proof of this criterion is given in \cite{Kw} in Theorem A1, and it is an analog of a theorem due to Kechris-Rosendal  \cite{KR}  in the context of the (injective) Fra\"{i}ss\'{e} theory.

Let $\g$ 
be a projective Fra\"{i}ss\'{e} family of finite $\mathcal{L}_0$-structures, for some language $\mathcal{L}_0$,
 with the limit $\mathbb{G}$.
Let $s$ be a binary relation symbol and let $\mathcal{L}_1$ be the language $\mathcal{L}_0\cup \{s\}.$ Define a class $\g^+$ of finite 
$\mathcal{L}_1$-structures  as follows:
\begin{equation*}
\begin{split}
\g^+= & \{ (A,s^A)\colon A\in\g  {\rm\ and \ there\ are\ }\phi\colon\mathbb{G}\to A {\rm\ and\ } f\in {\rm Aut}(\mathbb{G}) { \rm\ such\ that }\\
 &\phi\colon (\mathbb{G}, {\rm graph}(f))\to (A,s^A) { \rm\ is\ an\ epimorphism }\},
\end{split}
\end{equation*}
where $ {\rm graph}(f)$ is viewed as a closed relation on $\mathbb{G}$: $ {\rm graph}(f)(x,y)$ if and only if $f(x)=y$.

As in Subsection \ref{frase}, say that $\g^+$ has the joint projection property (the JPP) if and only if for every $(A,s^A),(B,s^B)\in \g^+$ there is $(C,s^C)\in\g^+$ and epimorphisms from  $(C,s^C)$ onto $(A,s^A)$ and from $(C,s^C)$ onto $(B,s ^B).$

\begin{proposition}[\cite{Kw}]\label{abc}
The group ${\rm Aut}(\mathbb{G})$ has a dense conjugacy class if and only if $\g^+$ has the JPP.
\end{proposition}

The lemma below is a general fact of the projective Fra\"{i}ss\'{e} theory. 
\begin{lemma}\label{obs}
Let $\g$ be a projective Fra\"{i}ss\'{e} family with the limit $\mathbb{G}$.  
Then $(T,s^T)\in \g^+$ if and only if there is $S\in \g$ and there are epimorphisms $p_1\colon S\to T$ and $p_2\colon S\to T$ such that
 $s^T=\{(p_1(x),p_2(x))\colon x\in S\}$.
\end{lemma}

\begin{proof}
($\Leftarrow$) 
Let $S,p_1,p_2$ be as in the hypothesis.
Let $\phi\colon \mathbb{G}\to S$ be any epimorphism (it exists by the universality property (L1) ). Let $\phi_1=p_1\circ \phi$ and let $\phi_2=p_2\circ\phi$. Using the
projective ultrahomogeneity (L3),  get $f\in{\rm Aut}(\mathbb{G})$ such that $\phi_1\circ f=\phi_2$. Then $\phi_1\colon (\mathbb{G},{\rm graph}(f))\to (T,s^T)$ is an epimorphism. So $(T,s^T)\in \g^+$.

($\Rightarrow$)  Let $(T,s^T)\in \g^+$.
Let $\psi\colon (\mathbb{G},f)\to (T,s^T)$ be an epimorphism. Denote $\phi_1=\psi$ and $\phi_2=\phi_1\circ f$. Let  $X$ be the common refinement of the partitions $\phi_1^{-1}(T)$ and $\phi_2^{-1}(T).$
Let $\alpha\colon\mathbb{G}\to S$, $S\in\g$, be an epimorphism refining the partition $X$. 
Then $p_1\colon S\to T$  satisfying $\phi_1=p_1\circ\alpha$ and $p_2\colon S\to T$  satisfying $\phi_2=p_2\circ\alpha$ are as required.
\end{proof}

 Every fan in $\f$ is specified by its height 
 and its width. 
Recall that we assumed that all branches in a given fan  have the same length. The {\em height} of a fan is the number of  elements in a branch minus one (we do not count the root).
The {\em  width} of a fan is  the number of its branches. 
Let $T$ be a fan of height $k$ and width $n$.
If $b$ is a branch in a fan $T$ of height $k$, we denote by $b(j)$ the $j$-th element of $b$, where $j=0,1,2,\ldots,k$
($b(0)$ is the root). 
We say that a binary relation $s^T$ on $T$ is {\em surjective} if for every $t\in T$ there are $r,s\in T$ such that $s^T(t,r)$ and $s^T(s,t)$.
Let $s^T$ be a surjective relation on $T$. Let $b_1,b_2,\ldots, b_n$ be the list of all branches of $T$ and let $r_T$ be the root of $T$. 
We say that $(x_1,y_1)\in T^2$ is {\em $s^T$-adjacent} to $(x_0,y_0)\in T^2$ if and only if 
 $R^T(x_0, x_1)$, $R^T(y_0, y_1)$, $s^T(x_0,y_0)$, and $s^T(x_1,y_1)$.
We say that $(c,d)$ is {\em $s^T$-connected} to $(a,b)$ if and only if 
there is $l$ and $(x_0,y_0),(x_1,y_1),\ldots, (x_l,y_l)\in T^2$
such that  $(x_0,y_0)=(a,b)$, $(x_l,y_l)=(c,d)$, and 
for each $i$,
$(x_{i+1}, y_{i+1})$ is $s^T$-adjacent to $(x_i,y_i)$.

\begin{lemma}\label{char}
We have $(T,s^T)\in\f^+$ if and only if $s^T$ is surjective, $s^T(r_T,r_T)$, and for every $(x,y)$, whenever $s^T(x,y)$, $(x,y)$ is $s^T$-connected to $(r_T,r_T)$.
\end{lemma}

\begin{proof}
($\Leftarrow$) 
 We define $S,p_1,p_2$ as in Lemma \ref{obs}. 
Let $k$ be the height of $T$. For every $(x,y)$ such that $s^T(x,y)$ we pick a chain of  length $2k+2$  and denote it by
$b_{(x,y)}$. 
Let $S$ be the disjoint union of all chains $b_{(x,y)}$ with their roots identified.
Now we define $p_1$ and $p_2$. Fix $(x,y)$ such that $s^T(x,y)$.
Fix a sequence $(r_T,r_T)=(x_0,y_0), (x_1,y_1),\ldots, (x_l,y_l)=(x,y)$ witnessing that $(x,y)$ is $s^T$-connected to $(r_T,r_T)$.
We let $p_1(b_{(x,y)}(i))=x_i$ and $p_2(b_{(x,y)}(i))=y_i$, whenever $i\leq l$.
We let $p_1(b_{(x,y)}(i))=x$ and $p_2(b_{(x,y)}(i))=y$, whenever $i> l$.

($\Rightarrow$)  Let $(T,s^T)\in \f^+$ and let  $S,p_1,p_2$ be as in Lemma \ref{obs}. 
Clearly $s^T(r_T,r_T)$.
Take $(x,y)$ such that  $s^T(x,y)$ and let $s\in S$ be such that $(x,y)=(p_1(s),p_2(s))$. 
Let $b$ be a branch in $S$ connecting $r_S$ to $s$, i.e. $r_S=s_0=b(0),s_1=b(1),\ldots, s_l=b(l).$ 
Then the sequence $(r_T,r_T)=(p_1(s_0),p_2(s_0)),\\ (p_1(s_1),p_2(s_1)),\ldots, (p_1(s_l),p_2(s_l))=(x,y)$ witnesses  that  $(x,y)$ is 
$s^T$-conne-\\cted to $(r_T,r_T)$.
\end{proof}

\begin{proposition}
The family $\f^+$ has the JPP.
\end{proposition}

\begin{proof}
Let $(T_1,s^{T_1})$, $(T_2,s^{T_2})\in\f^+$. For the JPP, take $T$ to be the disjoint union of $T_1$ and $ T_2$ with their respective roots identified.
For $x,y\in T$ we let $s^T(x,y)$ if and only if either $x,y\in T_1$ and $s^{T_1}(x,y)$, or $x,y\in T_2$ and $s^{T_2}(x,y)$.
Then, using Lemma \ref{char}, we conclude that $(T,s^T)\in\f^+$. Moreover, $\phi_1\colon (T,s^T)\to (T_1,s^{T_1})$ 
such that $\phi_1\restriction T_1=\text{Id}_{T_1}$ and $\phi_1\restriction T_2$ is mapped onto the root,
and  $\phi_2\colon (T,s^T)\to (T_2,s^{T_2})$
such that $\phi_2\restriction T_2=\text{Id}_{T_2}$ and $\phi_2\restriction T_1$ is mapped onto the root, are epimorphisms.

\end{proof}

\subsection{Simplicity of $\mathbf{H(L)}$}

A group is {\em simple} if it has no non-trivial proper normal subgroups. 
In this subsection, we show that the homeomorphism group of the Lelek fan, $H(L),$ is simple. Anderson \cite{A}  gave a criterion for  groups of homeomorphisms that implies the simplicity. Anderson's criterion is satisfied for instance by the homeomorphism group of the Cantor set, 
the homeomorphism group of the universal curve, or by the  group of all
 orientation-preserving homeomorphisms of $S^2.$ As we will see, a modification of that criterion applies to $H(L)$. 

There are various recent results concerning simplicity of topological groups, Tent-Ziegler \cite{TZ} showed that the isometry group of the bounded  Urysohn metric space  is simple, 
Macpherson-Tent \cite{MT} 
gave a general result on simplicity of automorphism groups of  countable ultrahomogeneous structures whose
 classes of finite substructures  have the free
amalgamation property.  This last result was later generalized by Tent-Ziegler \cite{TZ2}, who also showed 
that the isometry group of the Urysohn space modulo
the normal subgroup of bounded isometries is a simple group.

Recall from Subsection \ref{notation} that $E$ denotes the set of endpoints of $L$, 
$C$ is the Cantor set,  $F$ is the Cantor fan, and
 $\pi_1\colon F\setminus\{v\}\to C$,  $\pi_2\colon F\to [0,1]$ are projections. 
Let $v$ denote the top of $L$.
 Define 
 \[
 \K=\{K \subseteq L\colon \text{ both } K  \text{ and } (L\setminus K )\cup\{v\} \text{ are closed and different from } L \}.
 \]
 
 The properties listed below follow immediately from the definition of $\K $.
 \begin{remark}\label{rem}\label{easy}
 \begin{itemize}
 \item[(1)] Let $K \in \K $. Then for any $e\in E$, we have either $[v,e]\subseteq K $ or $[v,e]\subseteq (L\setminus K )\cup\{v\}$. 
  \item[(2)] Whenever $K \in \K $ and $g\in H(L)$, then  $g(K )\in \K $.
    \item[(3)] If $K \in \K $, then $K \setminus\{ v\}$ is an open non-empty set in $L$. Moreover $\{K \setminus\{ v\},
 L\setminus K \}$ is a clopen decomposition of $L\setminus\{ v\}$.
\item[(4)] If $K \in \K $, then $(L\setminus K )\cup\{v\}\in \K $. If $K ,K '\in \K $ are such that $K \cup K '\neq L$, then $K \cup K '\in \K $.  
  If $K ,K '\in \K $ are such that $K \cap K '\neq \{v\}$, then $K \cap K '\in \K $. 
     \item[(5)] If $X\subseteq C$ is a  clopen set such that $\pi_1^{-1}(X)\cap L$ and $\pi_1^{-1}(C\setminus X)\cap L$ are
     non-empty, then $(\pi_1^{-1}(X)\cap L)\cup \{ v\}\in \K $.
  \end{itemize}
  \end{remark}
 
 Let $G^0$ denote the subgroup of $H(L)$ consisting of those $g\in H(L)$ that are the identity when restricted to some $K \in \K .$ We say that $g\in G^0$ is {\em supported} on $K \in \K $ if $g\restriction (L\setminus K )$ is the identity. For $K \in \K $ let $E(K )$ denote the set of endpoints of $K$.
 Observe that by Remark \ref{rem} part (1), $E\cap K=E(K)$.
 
 \begin{lemma}\label{lem1}
 The family $\K $ satisfies the following properties:
 \begin{itemize}
 \item[(1)] each $K\in \K $ is homeomorphic to $L$,
 \item[(2)] for every $h\neq{\rm Id}\in H(L)$  there is $K\in \K $ such that \\ $K\cap (h(K)\cup h^{-1}(K))=\{v\}$.
 \end{itemize}
 \end{lemma}
 \begin{proof}
(1) Let $K\in \K $. To show that $K$ is homeomorphic to $L$, it is enough to show that $E(K)$ is dense in $K$. Let $x\in K\setminus\{ v\}$.   There is a sequence $(e_i)$ of endpoints of $L$ that converges to $x$. By passing to a subsequence, we can assume that either every $e_i$ is in $K$, or every $e_i$ is in $L\setminus K$. 
 Since  $(L\setminus K)\cup \{v\}$ is closed and $x\neq v$, the latter possibility cannot be true. 
 Therefore, since 
$E\cap K=E(K)$, the sequence $(e_i)$ is a sequence of endpoints of $K$ and it converges to $x$.
The above argument shows that $E(K)$ is dense in $K\setminus\{ v\}$. However, since
$\overline{K\setminus\{ v\}}=K$, $E(K)$ is also dense in $K$.

(2) Since $E$ is dense in $L$,  there is $e\in E$  such that $h(e)\neq e$.
Consequently,
 $h([v,e])\cap [v,e]=\{ v\}$ 
and $h^{-1}([v,e])\cap [v,e]=\{ v\}$. 
Let $M_1,M_2,M_3\in \K $ be such that $M_1\cap M_2=\{v\}$, $M_1\cap M_3=\{v\}$, $e\in M_1$,  $h(e)\in M_2$, and $h^{-1}(e)\in M_3$.
 Let $K=h^{-1}(M_2)\cap M_1\cap h(M_3)$. Then $K\in \K $, $K\subseteq M_1$, $h(K)\subseteq M_2$, and  $h^{-1}(K)\subseteq M_3$.
Therefore $K\cap (h(K)\cup h^{-1}(K))=\{v\}$.

 \end{proof}
 
 For $K\in \K ,$ define the {\em  height} of $K$ to be ${\rm max}(\pi_2(K)).$
 We say that a sequence $(K_i)_{i\in\mathbb{Z}}$ of elements of $\K $ is a {\em $\beta$-sequence if (1) $\bigcup_{i\in\mathbb{Z}}K_i \in \K $ and $K_i\cap K_j=\{v\}$ for $i\neq j$, and (2) $\lim_{i\to \infty} {\rm ht}(K_i)=0=\lim_{i\to -\infty}{\rm ht}(K_i).$

 \begin{lemma}\label{lem2}
 For every $K\in \K $ there exist a $\beta$-sequence $(K_i)$ with $\bigcup K_i=K$ and $\rho_1,\rho_2\in G^0$ supported on $K$ such that
\begin{itemize}
\item[(1)] $\rho_1(K_i)=K_{i+1}$ for each $i$;
\item[(2)] $\rho_2\restriction K_0=\rho_1\restriction K_0,$ $\rho_2\restriction K_{2i}=\rho_1^{-2}\restriction K_{2i}$ for $i>0,$ and $\rho_2\restriction K_{2i-1}=\rho_1^2\restriction K_{2i-1}$ for $i>0;$
\item[(3)] if  $\phi_i\in G^0$ is supported on $K_i,$ for each $i$, then there exists $\phi\in G^0$ supported on $K$ such that $\phi\restriction K_i=\phi_i \restriction K_i$ for every $i.$
\end{itemize}
 \end{lemma}

 \begin{proof}
 Given $K\in \K $, we first inductively construct a sequence $(K'_i)_{i\in\mathbb{N}}$ of elements of $\K $  such that $\bigcup_{i\in\mathbb{N}}K'_i =K$,  $K'_i\cap K'_j=\{v\}$ for $i\neq j$, and
  $\lim_{i\to \infty} {\rm ht}(K'_i)=0$.
Fix a compatible metric on the Cantor set $C$ such that ${\rm diam}(C)\leq 1$.

To construct $K'_0$, pick $e\in E(K)$ such that $\pi_2(e)<2^{-1}$.
Let $X\subseteq C$ be a clopen such that $\pi_1(e)\in X$ and 
${\rm ht}(M)<2^{-1}$, where $M=(\pi_1^{-1}(X)\cap L)\cup\{ v\}$.  Let $K'_0=(K\setminus M)\cup\{v\}$.
Then $K'_0\in \K $ and $K\setminus K'_0\neq\emptyset$ since $e\in K\setminus K'_0$.
Note that  ${\rm ht}((K\setminus K'_0)\cup\{v\})= {\rm ht}(M)<2^{-1}$.

  Suppose that  we constructed 
  $K'_0, K'_1,\ldots, K'_n$ such that  (a) for every $i\neq j$, $i,j\leq n$, $K'_i\cap K'_j=\{v\}$ and
  $K\setminus \bigcup_{j\leq i} K'_j\neq\emptyset$,
  (b) for every $i\leq n$, ${\rm ht}(K'_i)<2^{-i}$ and  
  ${\rm ht}((K\setminus\bigcup_{j\leq i} K'_j)\cup\{v\})<2^{-(i+1)}$,  and 
  (c)   for every $i\leq n$, 
${\rm diam}(\pi_1(K\setminus \bigcup_{j\leq i} K'_j))<2^{-(i-1)}$.
 
  Now we construct $K'_{n+1}$ such that conditions (a), (b), and (c), with $n$ replaced by $n+1$, 
  are fulfilled: Using that $(K\setminus \bigcup_{j\leq n} K'_j)\cup\{ v\}\in \K $ and consequently
  $K\setminus \bigcup_{j\leq n} K'_j$ is open in $L$,
  pick $e\in E(K\setminus \bigcup_{j\leq n} K'_j)$ such that $\pi_2(e)<2^{-(n+2)}$.
  Further let $X\subseteq C$ be a clopen such that $\pi_1(e)\in X$ and 
${\rm ht}(M)<2^{-(n+2)}$, where $M=(\pi_1^{-1}(X)\cap L)\cup\{ v\}$.  
By shrinking $M$ if necessary, we can assume that $(M\cap K)\cup \bigcup_{j\leq n} K'_j\neq K$ and 
${\rm diam}(\pi_1(M\setminus\{v\}))<2^{-n}.$
Let $K'_{n+1}=(K\setminus (\bigcup_{j\leq n} K'_j\cup M))\cup\{v\}$.
Then $K'_{n+1}\in \K $ is as required. In particular,  $K\setminus \bigcup_{j\leq n+1} K'_j\neq\emptyset$,
${\rm ht}((K\setminus\bigcup_{j\leq n+1} K'_j)\cup\{v\})= {\rm ht}(M)<2^{-(n+2)}$ and 
${\rm diam}(\pi_1(K\setminus \bigcup_{j\leq n+1} K'_j))\leq {\rm diam}(\pi_1(M\setminus\{v\}))<2^{-n}.$

The sequence $(K_i)_{i\in\mathbb{Z}}$ such that $K_0=K'_0$, $K_{-i}=K'_{2i}$  for $i=1,2,\ldots$,
and $K_i=K'_{2i-1}$  for $i=1,2,\ldots$, is a $\beta$-sequence satisfying
  $\bigcup_{i\in\mathbb{Z}}K_i =K$. 
  
  We first show that (3) holds. Let $\phi_i$ be as in the assumptions. Let $\phi$ be such that $\phi\restriction K_i=\phi_i\restriction K_i$ and let $\phi$ be equal to the identity
  outside $K$. We want to show that $\phi$ is a homeomorphism. Clearly $\phi$ is a bijection. Since $L$ is compact, it is enough to show that $\phi$ is continuous.
  Let $x\in L$. We show that $\phi$ is continuous at $x$. If $x\neq v$, then $x\in K_i\setminus\{v\}$ for some $i$, or $x\in L\setminus K.$ 
  Since each of $ K_i\setminus\{v\}$ and $L\setminus K$ is open,
  whenever $(x_n)$
   converges to $x$, then eventually $x_n\in K_i\setminus\{v\}$ for some $i$ or $x_n\in L\setminus K$, respectively. Therefore, eventually 
    $\phi(x_n)\in K_i\setminus\{v\}$ for some $i$, or $\phi(x_n)\in L\setminus K$, respectively. Since each $\phi_i$ is continuous, $\phi(x_n)$ converges to $\phi(x)$.
    Now let $x=v$ and let $(x_n)$ converge to $v$.
We show that $\phi(x_n)$ converges to $v=\phi(v)$.   
     Fix an open neighbourhood $U$ of $v$. Since ${\rm ht}(K_i)\to 0$ both for $i\to \infty$ and for $i\to -\infty,$ we can find $i_0>0$ such that when $i>i_0$ or $i<-i_0$, then $K_i\subseteq U$.
By continuity of $\phi_i,$ find $n_0$  such that whenever $n>n_0$ and $x_n$ is in one of $K_i$, $-i_0\leq i\leq i_0$, or in $L\setminus K$, then   $\phi(x_n)=\phi_i(x_n)\in U$, or $\phi(x_n)=x_n\in U$ respectively.
Then  since for each $i$, $\phi_i(K_i)\subseteq K_i$, whenever $n>n_0$, we have $\phi(x_n)\in U$. This shows the continuity of $\phi$ at $v$. 

To show (1) we let $\rho^i_1\colon K_i\to K_{i+1}$ to be any homeomorphism, which exists as all $K_i$'s  are homeomorphic to the Lelek fan.
Let $\rho_1$ be such that $\rho_1\restriction K_i=\rho^i_1$, $i\in\mathbb{Z}$, and let $\rho_1$ be the identity outside $K$. Then  similarly as in the proof of (3), we show that $\rho_1$ is a 
homeomorphism of $L$.

Having defined $\rho_1$, we set $\rho_2$ on each $K_i$, $i\geq 0$, as in (2),
and we let $\rho_2$ to be the identity otherwise.  Then again  similarly as in the proof of (3), we show that $\rho_2$ is a 
homeomorphism of $L$.


 \end{proof}

\begin{remark}
Anderson \cite{A} showed that whenever $G$ is a group of homeomorphisms of a space $X$, and there exists a family $\K $  of closed sets that satisfies
conditions similar to those given in Remark \ref{easy} and in Lemmas \ref{lem1} and \ref{lem2}, then $G$ is a simple group.
He assumes that sets $(K_i)$ in the definition of a $\beta$-sequence are disjoint.
Moreover,  he assumes that for every open non-empty set $U\subseteq X$
there exists $K\in \K $ such that $K\subseteq U$, which is false in our situation. Nevertheless, it is enough to substitute it
by the condition  (2) of Lemma \ref{lem1}.

\end{remark}

\begin{theorem}\label{simple}
The group of all homeomorphisms of the Lelek fan, $H(L),$ is simple.
\end{theorem}

{\rm The proof of Theorem \ref{simple} will go along the lines of the proof  of simplicity of homeomorphism groups studied by Anderson. 
 We sketch it here for the reader's convenience, and for the details we refer   to \cite{A}.

We need the following lemma (analogous to Theorem I in \cite{A}). }

\begin{lemma}\label{simlem}
Let $h\neq{\rm Id}\in H(L)$. Then every $g\in G^0$ is the product of four conjugates of $h$ and $h^{-1}$ (appearing alternately).
\end{lemma}

\begin{proof}
Since every two elements of $\K $ are homeomorphic via a homeomorphism of $L$, it is enough to show that there exists $K_0\in \K $ such that for any $g_0\in G^0$ supported on $K_0$, 
$g_0$  is the product of four conjugates of $h$ and $h^{-1}$.

By Lemma \ref{lem1} (2), there is
 $K\in \K $  such that $K\cap (h(K)\cup h^{-1}(K))=\{v\}$. Let $(K_i)$ be a $\beta$-sequence such that $\bigcup_i K_i=K$ and let
$\rho_1 $ and $\rho_2$ be as in (1) and (2) of Lemma \ref{lem2}. We show that $K_0$ is as required.
Let $g_0\in G^0$ be supported on $K_0.$
For $i\geq 0$, let $\phi_i=\rho_1^i g_0 \rho_1^{-i}$ and let $\phi_i$ be the identity on $K_i$, when $i<0$. 
Take $\phi$ as in (3) of Lemma \ref{lem2}.
Take $f=h^{-1} \phi^{-1} h \phi$. Note that $f$ is supported on $K\cup h^{-1}(K)$, $f\restriction K=\phi\restriction K$, 
and $f\restriction (h^{-1}(K))=(h^{-1} \phi^{-1} h)\restriction (h^{-1}(K))$.
Take $\rho=h^{-1} \rho_2 h \rho_1^{-1}$. Note that $\rho$ is supported on $K\cup h^{-1}(K)$, $\rho\restriction K=\rho_1^{-1}\restriction K$, 
and $\rho\restriction (h^{-1}(K))=(h^{-1} \rho_2 h)\restriction (h^{-1}(K))$.
Let $w=\rho^{-1} f^{-1}\rho f$. Then 
$w=(\rho^{-1}\phi^{-1} h^{-1}\phi\rho)(\rho^{-1} h\rho)(h^{-1})(\phi^{-1} h\phi)$, 
therefore  is a product of four conjugates of $h$ and $h^{-1}.$
Unraveling definitions of $\phi$, $f$, $\rho$, and $w$, as it is done in \cite{A},
we get that $w=g_0$.

\end{proof}

\begin{proof}[Proof of Theorem \ref{simple}]
Let $g\in H(L)$ and let $h\in H(L)$, $h\neq{\rm Id}$. We show that $g$ is the product of 8 conjugates of $h$ and $h^{-1}$.
This will immediately imply that $H(L)$ is simple.

Let $K\in \K $ be such that $g(K)\cap K=\{v\}$ and $g(K)\cup K\neq L$. Take $\alpha\in H(L)$ such that $\alpha\restriction K=g\restriction K$,
$\alpha\restriction g(K)=g^{-1}\restriction g(K)$,
and $\alpha $ is equal to the identity outside $g(K)\cup K$.
Notice that $\alpha,(\alpha^{-1}g)\in G^0$ and $g=\alpha(\alpha^{-1}g)$. By Lemma \ref{simlem}, $g$ is the product of 8 conjugates of $h$ and $h^{-1}$.
\end{proof}

\begin{remark}
 As in \cite{A}, one can modify the proof of Theorem \ref{simple}, to show that whenever $g\in H(L)$ and  $h\in H(L)$, $h\neq {\rm Id}$,
then $g$ is the product of 6 conjugates of $h$ and $h^{-1}$.
\end{remark}

\thanks{{\noindent \bf Acknowledgments. }{\rm  A large portion of this work was done during the trimester program on `Universality and Homogeneity' at the Hausdorff Research Institute for Mathematics in Bonn.
We would like to thank the organizers: Alexander Kechris, Katrin Tent, and Anatoly Vershik for the opportunity to participate in the program.
We also would like to thank the anonymous referee for numerous detailed suggestions that considerably helped us to improve the presentation of the paper.
}

\end{document}